\documentclass[12pt]{amsart}
\usepackage{times}
\usepackage{amssymb}
\usepackage{amsmath}
\usepackage{comment}
\usepackage{tikz,xcolor,tasks}
\begin{document}
\numberwithin{equation}{section}

\def\1#1{\overline{#1}}
\def\2#1{\widetilde{#1}}
\def\3#1{\widehat{#1}}
\def\4#1{\mathbb{#1}}
\def\5#1{\frak{#1}}
\def\6#1{{\mathcal{#1}}}

\newcommand{\de}{\partial}
\newcommand{\R}{\mathbb R}
\newcommand{\al}{\alpha}
\newcommand{\tr}{\widetilde{\rho}}
\newcommand{\tz}{\widetilde{\zeta}}
\newcommand{\tv}{\widetilde{\varphi}}
\newcommand{\tO}{\widetilde{\Omega}}
\newcommand{\hv}{\hat{\varphi}}
\newcommand{\tu}{\tilde{u}}
\newcommand{\usc}{{\sf usc}}
\newcommand{\tF}{\tilde{F}}
\newcommand{\debar}{\overline{\de}}
\newcommand{\Z}{\mathbb Z}
\newcommand{\C}{\mathbb C}
\newcommand{\Po}{\mathbb P}
\newcommand{\zbar}{\overline{z}}
\newcommand{\G}{\mathcal{G}}
\newcommand{\So}{\mathcal{U}}
\newcommand{\Ko}{\mathcal{K}}
\newcommand{\U}{\mathcal{U}}
\newcommand{\B}{\mathbb B}
\newcommand{\oB}{\overline{\mathbb B}}
\newcommand{\Cur}{\mathcal D}
\newcommand{\Dis}{\mathcal Dis}
\newcommand{\Levi}{\mathcal L}
\newcommand{\SP}{\mathcal SP}
\newcommand{\Sp}{\mathcal Q}
\newcommand{\Ma}{\mathcal M}
\newcommand{\Co}{\mathcal C}
\newcommand{\Hol}{{\sf Hol}(\mathbb H, \mathbb C)}
\newcommand{\Aut}{{\sf Aut}(\mathbb D)}
\newcommand{\D}{\mathbb D}
\newcommand{\oD}{\overline{\mathbb D}}
\newcommand{\oX}{\overline{X}}
\newcommand{\loc}{L^1_{\rm{loc}}}
\newcommand{\loci}{L^\infty_{\rm{loc}}}
\newcommand{\la}{\langle}
\newcommand{\ra}{\rangle}
\newcommand{\thh}{\tilde{h}}
\newcommand{\N}{\mathbb N}
\newcommand{\kd}{\kappa_D}
\newcommand{\Hr}{\mathbb H}
\newcommand{\ps}{{\sf Psh}}
\newcommand{\tg}{\widetilde{\gamma}}

\newcommand{\subh}{{\sf subh}}
\newcommand{\harm}{{\sf harm}}
\newcommand{\ph}{{\sf Ph}}
\newcommand{\tl}{\tilde{\lambda}}
\newcommand{\ts}{\tilde{\sigma}}

\def\v{\varphi}
\def\Re{{\sf Re}\,}
\def\Im{{\sf Im}\,}

\def\e{{\sf e}}

\def\au{{\underline a}}

\def\dist{{\rm dist}}
\def\const{{\rm const}}
\def\rk{{\rm rank\,}}
\def\id{{\sf id}}
\def\aut{{\sf aut}}
\def\Aut{{\sf Aut}}
\def\CR{{\rm CR}}
\def\GL{{\sf GL}}
\def\U{{\sf U}}

\def\la{\langle}
\def\ra{\rangle}
\def\Ha{\mathbb H}

\newtheorem{theorem}{Theorem}[section]
\newtheorem{lemma}[theorem]{Lemma}
\newtheorem{proposition}[theorem]{Proposition}
\newtheorem{corollary}[theorem]{Corollary}

\theoremstyle{definition}
\newtheorem{definition}[theorem]{Definition}
\newtheorem{example}[theorem]{Example}

\theoremstyle{remark}
\newtheorem{remark}[theorem]{Remark}
\numberwithin{equation}{section}

\makeatletter
\@namedef{subjclassname@2020}{\textup{2020} Mathematics Subject Classification}
\makeatother

\title[Invariant subspaces for finite index shifts]{Invariant subspaces for finite index shifts in Hardy spaces}
\author[F. Bracci]{Filippo Bracci}
\author[E. A. Gallardo-Guti\'errez]{Eva A. Gallardo-Guti\'errez}

\address{F. Bracci\newline
Dipartimento Di Matematica\newline
Universit\`{a} di Roma \textquotedblleft Tor Vergata\textquotedblright\ \newline
Via Della Ricerca Scientifica 1, 00133 \newline
Roma, Italy}
\email{fbracci@mat.uniroma2.it}
\thanks{The first author is partially supported by PRIN (2022) Real and Complex Manifolds: Geometry and holomorphic dynamics Ref:2022AP8HZ9, by GNSAGA of INdAM and  by the MUR Excellence Department Project 2023-2027 MatMod@Tov CUP:E83C23000330006 awarded to the Department of Mathematics, University of Rome Tor Vergata}

\address{E. A. Gallardo-Guti\'errez \newline Departamento de An\'alisis Matem\'atico y\newline
Matem\'atica Aplicada \newline
Facultad de CC. Matem\'aticas,
\newline Universidad Complutense de
Madrid, \newline
 Plaza de Ciencias 3, 28040 Madrid,  Spain
 \newline
and Instituto de Ciencias Matem\'aticas ICMAT(CSIC-UAM-UC3M-UCM),
\newline Madrid,  Spain }
\email{eva.gallardo@mat.ucm.es}
\thanks{Second author is partially supported by Plan Nacional  I+D grant no. PID2022-137294NB-I00, Spain.
\newline
The authors also acknowledge support from ``Severo Ochoa Programme for Centres of Excellence in R\&D'' (CEX2019-000904-S \& CEX2023-
001347) of the
Ministry of Economy and Competitiveness of Spain and by the European Regional Development
Fund and from the Spanish
National Research Council, through the ``Ayuda extraordinaria a
Centros de Excelencia Severo Ochoa'' (20205CEX001).
}


\begin{abstract}
Let $\mathbb H$ be the finite direct sums of $H^2(\mathbb D)$. In this paper, we give a characterization of  the  closed subspaces of $\mathbb H$ which are invariant under the shift, thus obtaining a concrete Beurling-type theorem for the finite index shift. This characterization presents any such a subspace as the finite intersection, up to an inner function, of pre-images of a closed shift-invariant subspace of $H^2(\mathbb D)$ under  ``determinantal operators'' from $\mathbb H$ to  $H^2(\mathbb D)$, that is, continuous linear operators which intertwine the shifts and appear as determinants of matrices with entries given by bounded holomorphic functions. With simple algebraic manipulations we provide a direct proof that every invariant closed subspace of codimension at least two sits into a non-trivial closed invariant subspace. As a consequence every contraction with finite defect has a nontrivial closed  invariant subspace.
\end{abstract}

\maketitle

\tableofcontents

\section{Introduction}

Let $H^2(\D)$ be the classical Hardy space on the unit disc $\D$. Beurling's theorem \cite{Beu, Beu2} states that every closed subspace of $H^2(\D)$  invariant under the shift  are of the form $\phi H^2(\D)$, where $\phi$ is either identically $0$ or an inner function.

Beurling's Theorem is a cornerstone in the theory of function spaces. It has been extended to broader classes of operators, with significant advancements achieved through the development of sophisticated tools in related fields such as Harmonic Analysis, Function Theory, and finite-dimensional Linear Algebra within the framework of Operator Theory (see, {\sl e.g.}, \cite{Duren, CiRo, RoRo, ARR, ARS, Ri, RR, Ni, Ni2}).

The importance of understanding the lattice of invariant subspaces for the shift operators, besides its intrinsic interest, lies in the universality property of their adjoints (see, {\sl e.g.}, \cite[Section~1.5]{RoRo}). Specifically, every continuous endomorphism of defect $\delta$ in a Hilbert space can be modeled (up to a constant) through the restriction to an invariant subspace of any backward shift of index at least $\delta$.
From Beurling's theorem, it is straightforward to understand the lattice of invariant subspaces of the shift of index one. Via the universality property, this provides an affirmative solution to the invariant subspace problem for bounded linear operators with defect $\leq 1$.

Beurling's Theorem can be viewed as stating that every non-trivial closed invariant subspace for the shift in $H^2(\D)$ is the image of an isometry which commutes with the shift. From this perspective, Lax \cite{Lax} (for finite index shifts), along with Halmos \cite{Halmos} and Rovnyak \cite{Rov} (for infinite index shifts), extended Beurling's theorem to what is now known as the Beurling-Lax Theorem (see also \cite[Section~1.12]{RoRo}). They proved that every closed subspace of a Hilbert space, invariant under a shift, is the image of a quasi-isometry that commutes with the shift.

While the Beurling-Lax Theorem fully characterizes closed invariant subspaces for the shift in any Hilbert space, its generality and abstraction can pose challenges when working with specific cases.

For instance, directly from this theorem, it seems difficult to obtain information about maximal invariant subspaces for the shift, even for finite index shifts. In fact, by a result of Atzmon \cite{Atz}, it is known that maximal invariant subspaces for finite index shifts have codimension one---this also follows from Guo, Hei and Hou \cite{Guo}, who proved a similar statement for the restriction of the multiplication on the Bergman space to a finite index invariant subspace, and hence the result follows by the ``universality'' of the lattice of invariant subspaces of such an operator \cite[Corollary~3.4]{ABFP}. However those proofs are non constructive and could be very difficult to generalize to infinite index shifts.

In this paper, we present a novel approach to describing closed invariant subspaces for the finite index shift on the direct sum of Hardy spaces. This approach proves to be ``effective' in the sense that, starting from this characterization, we can easily demonstrate through simple algebraic manipulations that the only maximal invariant subspaces have codimension one.

In order to state our main result, we need to introduce some notations (see Section~\ref{Sec:deter-operator} for details and precise statements). Let $d$ be a positive integer and let $\Ha$ be the Hilbert space given by the direct sum of $d$ copies of the Hardy space $H^2(\D)$. The elements of $\Ha$ are $d$-tuple of functions in $H^2(\D)$, and we write $F\in\Ha$ if $F(z)=f_1(z)\e_1+\ldots+f_d(z)\e_d$, $z\in \D$ and $f_j\in H^2(\D)$. Let $S:\Ha \to \Ha$ be the {\sl shift operator}, defined as $S(F)(z)=zF(z)=zf_1(z)\e_1+\ldots+zf_d(z)\e_d$, $z\in \D$. We write $S_1:=S$  if $d=1$.

The basic observation is that if $R: \mathbb H\to H^2(\D)$ is a bounded linear operator such that $R\circ S=S_1\circ R$ and $\phi$ is a inner function, then $M:=R^{-1}(\phi H^2(\D))$ is an $S$-invariant closed subspace of $\mathbb H$. Our main result is that every closed $S$-invariant subspace of $\mathbb H$ are, up to an inner function, finite intersections of subspaces like $M$. From this perspective, our approach also makes sense  for $d=\infty$. Indeed, the basic question is whether the same conclusion holds in such a case (possibly with infinitely many intersections). In fact, we also specify the form of the operators $R$, which turns out to be what we call ``determinantal operators'', and we now describe them in detail.

Let $1\leq m\leq d$ be an integer number and let $A=(a_{jk})$ be a $m\times m$ matrix whose entries are bounded holomorphic functions in $\D$ (that is, elements of $H^\infty(\D)$). Let $1\leq s_1<\ldots <s_m \leq d$. Let $j\in \{1,\ldots, m\}$. A {\sl determinantal operator} is any linear operator $L:\Ha \to H^2(\D)$ of the form
\[
L(f_1\e_1+\ldots+f_d\e_d):=\det \left(\begin{matrix} a_{11} & \ldots &  a_{1m}\\ \vdots & \vdots & \vdots\\ a_{j-1,1} & \ldots &  a_{j-1,m}\\ f_{s_1} & \ldots & f_{s_m} \\ a_{j+1,1} & \ldots &  a_{j+1,m}\\ \vdots & \vdots &  \vdots\\  a_{m1} & \ldots &  a_{mm}\\  \end{matrix} \right).
\]
Let $V\subseteq H^2(\D)$ be a closed $S_1$-invariant subspace. Hence, there exists $\varphi$ which is either identically zero or a inner function, such that $V=\varphi H^2(\D)$. A {\sl determinantal brick $Q_\varphi$ based on $\varphi$} is any set of the form
\[
Q_\varphi:=L^{-1}(\varphi H^2(\D))=\{F\in \Ha: L(F)\in \varphi H^2(\D)\},
\]
where $L$ is a determinantal operator. The set $Q_\varphi$ is a closed $S$-invariant subspace of $\Ha$. A {\sl determinantal space $\mathcal Q_\varphi$ based on $\varphi$} is the intersection of a finite number of determinantal bricks based on $\varphi$, that is,
\[
\mathcal Q_\varphi=Q^1_\varphi\cap \ldots \cap Q^n_\varphi,
\]
where  $Q^j_\varphi$ is a determinantal brick based on $\varphi$, $j=1,\ldots, n$. With these definitions at hand, the main result of the paper can then be stated as follows:

\begin{theorem}\label{Thm:main-intro}
Let $N\neq\{0\}$ be a closed subspace of $\Ha$. Then $N$ is $S$-invariant if and only if there exist inner functions $\varphi, \phi$ and determinantal subspaces $\mathcal Q_\varphi$ and $\mathcal Q_0$ such that either
\[
N=\phi \left( \mathcal Q_\varphi\cap \mathcal Q_0 \right),
\]
or
\[
N=\phi  \mathcal Q_\varphi.
\]
\end{theorem}
In the previous statement, it might happen that either $\mathcal Q_\varphi=\Ha$ or $\mathcal Q_0=\Ha$ (or both, in case $N=\Ha$). Note also that, in case $d=1$, the previous theorem reduces to the classical Beurling's Theorem. Indeed, in that case, the only determinantal operator is the identity operator $L(f_1\e_1)=f_1$ and hence the determinantal bricks are of the form $Q_\varphi=\varphi H^2(\D)$ and $Q_0=\{0\}$.

Theorem~\ref{Thm:main-intro} follows at once from Theorem~\ref{Thm:subspace-includL} (where a more precise statement is given). The starting point of the proof is  the Beurling-Lax theorem (see Section~\ref{Sec:Beurling-Lax}): any closed $S$-invariant subspace of $\Ha$ is the image of a so-called $S$-inner operator of $\Ha$. Such an operator is actually given by the a $d\times d$ matrix with entries in $H^\infty(\D)$. This is the matrix from which we define the determinantal subspaces associated to a closed $S$-invariant subspace. The two different cases in the theorem depends on whether the determinant of such a matrix is identically zero or not.

The concrete form given by Theorem~\ref{Thm:main-intro} allows us to prove directly the following result (cfr. \cite{Atz}):

\begin{theorem}\label{Thm:no-maximal-intro}
Let $N\subsetneq \Ha$ be a closed $S$-invariant subspace. Suppose that $\dim N^\perp\geq 2$. Then there exists a closed $S$-invariant subspace $M\subsetneq \Ha$ such that $N\subsetneq M$.
\end{theorem}

This theorem follows immediately from Proposition~\ref{Prop:NA-no-maximal}, Proposition~\ref{Prop:NA-no-maximal2} and Theorem~\ref{Thm:subspace-includL}.

As we stated in the introduction, Theorem~\ref{Thm:no-maximal-intro}  allows to give an affirmative answer to the invariant subspace problem for contractions with finite defect (for an account history and circle of ideas related to the subspace invariant problem  see, {\sl e.g.}, \cite{ChPa, EL}). Namely, let $H$ be a separable Hilbert space, and let $T:H\to H$ be a contraction, namely, $\|T\|\leq 1$. The {\sl defect of $T$} is
\[
\delta(T):=\dim \overline{(I-T^\ast T)H}.
\]
Roughly speaking, $\delta(T)$ ``measures'' how much far from an isometry  $T$ is. By the universality of the  shifts (see, {\sl e.g.}, \cite[Section~1.5]{RoRo}), every contraction $T:H\to H$  with $\delta(T)\leq d$ is unitarily equivalent to the restriction of the backward shift $S^\ast$ to some of its invariant closed subspaces. By Theorem~\ref{Thm:no-maximal-intro}, the backward shift $S^\ast:\Ha\to \Ha$ does not have non-trivial minimal invariant closed subspaces of dimension greater than $1$. Hence, we have

\begin{corollary}
Let $H$ be a separable Hilbert space. Let $T:H\to H$ be a contraction such that $\delta(T)<\infty$. Then there exists a closed subspace $M\subsetneq H$, $M\neq \{0\}$ such that $T(M)\subseteq M$.
\end{corollary}

\medskip

\centerline{\bf Acknowledgements}

\medskip

Part of this work has been done while the first named author was ``Profesor Distinguido'' at the research institute \emph{Instituto de Ciencias Matem\'aticas (ICMAT),
	 Spain}. He wants sincerely to thank ICMAT for the possibility of spending time there and work in a perfect atmosphere.
	
\medskip

The authors were not aware of the paper \cite{Guo} when completing the first draft of this manuscript. They sincerely thank Prof. Alexandru Aleman and Prof. Stefan Richter for bringing this reference to their attention and for other useful comments on the paper.

\medskip

\centerline{\bf Conflict of interest and Data Availability}

On behalf of all authors, the corresponding author states that there is no conflict of interest.

No datasets were generated or analysed during the current study.

\section{Preliminaries}

Let $d$ be a positive integer and let $\Ha$ be the Hilbert space given by the direct sum of $d$ copies of the Hardy space $H^2(\D)$. It is convenient to write
\[
\Ha=H^2(\D)\e_1\oplus \ldots \oplus H^2(\D)\e_d,
\]
where $\e_j=(0,\ldots, 0,1,0,\ldots, 0)$, with $1$ in the $j$-th position, $j=1,\ldots, d$.

The elements of $\Ha$ are given by $\sum_{j=1}^d f_j\e_j$ where $f_j\in H^2(\D)$, $j=1,\ldots, d$. The vector space $\Ha$ is a (separable) Hilbert space with Hermitian product given by
\[
\la \sum_{j=1}^d f_j\e_j ,  \sum_{j=1}^d g_j\e_j\ra:=\sum_{j=1}^d \la f_j, g_j\ra_{H^2(\D)}.
\]
The shift operator $S:\Ha\to \Ha$ is defined as
\[
S(f_1(z)\e_1+\ldots+f_d(z)\e_d):=zf_1(z)\e_1+\ldots+zf_d(z)\e_d.
\]
As a convenient notation, throughout the paper the shift operator for $d=1$, {\sl i.e.}, $\Ha=H^2(\D)$, will be denoted by $S_1$, {\sl i.e.}, $S_1(f(z)):=zf(z)$, $f\in H^2(\D)$.

As customary, if $X$ is a Hilbert space and $T:X\to X$ is a bounded linear operator, we say that a closed subspace $Y\subset X$ is $T$-invariant if $T(Y)\subseteq Y$.

Given a bounded function $g\in H^\infty(\D)$, if $V\subset \Ha$ is a closed subspace, we denote by $gV$ the image of $V$ via the continuous linear operator of $\Ha$ into $\Ha$ given by
\[
 f_1\e_1+\ldots + f_d\e_d\mapsto g f_1\e_1+\ldots + g f_d\e_d.
\]
In general, $gV$ is not closed. However, if $g$ is {\sl inner} (that is, $g:\D\to \D$ is a holomoprhic function such that $\lim_{r\to 1^-}|g(r e^{i\theta})|=1$ for a.e. $\theta\in[0,2\pi]$), then the operator $\Ha \ni F\mapsto gF$ is an isometry and if $V$ is closed subspace of $\Ha$, then so is $gV$.

The algebra of inner functions is well studied (see, {\sl e.g.}, \cite{CiRo} and references therein). As customary, we say that a inner function $\phi$ is {\sl invertible} or {\sl constant} if $\phi(z)=\lambda$ for all $z\in\D$ and some   unimodular constant $\lambda$. We say that an inner function $\phi$ divides a inner function $\varphi$ provided $\frac{\phi}{\varphi}$ is  a inner function. Given any family $\{\phi_j\}_{j\in\mathcal I}$ of inner functions, there exists a inner function $\phi$, unique up to multiplication by a unimodular constant, such that $\phi$ divides $\phi_j$ for all $j\in\mathcal I$ and, if $\varphi$ is a inner function that divides all elements of $\{\phi_j\}_{j\in\mathcal I}$, then $\varphi$ divides $\phi$. The function $\phi$ is called the {\sl greatest common divisor} of $\{\phi_j\}_{j\in\mathcal I}$---strictly speaking,  $\phi$ is not unique, which introduces some ambiguity in the choice. However, since it is unique up to multiplication by a unimodular constant, this ambiguity is inconsequential. Two inner functions are {\sl coprime} if their greatest common divisor is $1$.

By the classical Beurling theorem \cite{Beu} (or see, {\sl e.g.}, \cite{CiRo, Duren}), a closed subspace $V$ of $H^2(\D)$ is $S_1$-invariant if and only if there exists $\phi\in H^\infty(\D)$  an inner function, such that $V= \phi H^2(\D)$.

The following lemma is well known, and we omit the proof:
\begin{lemma}\label{Lem:dim-one}
Let $\varphi$ be an inner function. Suppose that $\dim (\varphi H^2(\D))^\perp\geq 2$. Then there exists a not constant inner function $\tilde\varphi$ such that $\tilde\varphi$ divides $\varphi$ and $\varphi$ does not divide $\tilde\varphi$. In particular,  $\varphi H^2(\D)\subsetneq \tilde\varphi H^2(\D)\subsetneq H^2(\D)$.
\end{lemma}

We will make use of the inner-outer factorization of a holomorphic function  (see, {\sl e.g.}, \cite{Duren}):

\begin{definition}
For any $f\in H^2(\D)$, $f\not\equiv 0$, we write $f=I(f) O(f)$ for the inner-outer factorization, where $I(f)$ is inner and $O(f)\in H^2(\D)$ is outer.
\end{definition}

For the aim of this paper, the only relevant property of an outer function is that it is $S_1$ cyclic. That is, if $O\in H^2(\D)$ is an outer function,
\[
\overline{\hbox{span}\{S^n(O): n\in\N\}}^{H^2(\D)}=\overline{\{p O: p \hbox{ polynomial}\}}=H^2(\D).
\]

In the following sections, we will need the lemma below, for which we provide a proof due to the lack of a direct reference.

\begin{lemma}\label{Lem:coprime-corona}
Let $\phi_1, \ldots, \phi_m$ be inner functions. Let $\phi$ be the (inner function) greatest common divisor of $\phi_1, \ldots, \phi_m$. Then there exist $m$ sequences $\{h^j_n\}_{n\in\N}\subset H^2(\D)$, $j=1,\ldots, m,$ such that $\{\sum_{j=1}^m h^j_n\phi_j\}$ converges to $\phi$ in $H^2(\D)$.
\end{lemma}
\begin{proof}
We have that $V:=\sum_{j=1}^m\phi_j H^2(\D)\subset \phi H^2(\D)$. Thus $\overline{V}^{H^2(\D)}\subseteq \phi H^2(\D)$. Since $V$ is $S_1$-invariant, by Beurling Theorem there exists a inner function $\varphi$ such that $\overline{V}^{H^2(\D)}=\varphi H^2(\D)$. Hence, $\varphi$ divides  $\phi_1, \ldots, \phi_m$. In particular, $\varphi$ divides $\phi$. Thus,
\[
\overline{V}^{H^2(\D)}\subseteq\phi H^2(\D)\subseteq \varphi H^2(\D)=\overline{V}^{H^2(\D)},
\]
and the result follows.
\end{proof}

\section{Closed $S$-invariant spaces  associated to matrices}\label{Sec:deter-operator}

In this section we associate to every $d\times d$ matrix whose entries are in $H^\infty(\D)$ a closed $S$-invariant subspace of $\Ha$. We start with some general preliminaries.

\begin{definition}
For  positive integers $m, n$, we denote by $\mathcal M_\infty(m\times n)$ the set of all $m\times n$ matrices whose entries are in $H^\infty(\D)$. In other words, $B\in \mathcal M_\infty(m\times n)$ if there exist $b_{jk}\in H^\infty(\D)$, $j=1,\ldots, m$, $k=1,\ldots, n$, such that
\[
B=\left(\begin{matrix} b_{11} & \ldots &  b_{1n}\\
\ldots & \ldots & \ldots \\   b_{m1} & \ldots & b_{mn}\\  \end{matrix} \right).
\]
\end{definition}

For square matrices we can associate certain operators which will be crucial for our characterization of closed $S$-invariant spaces:

\begin{definition}\label{Def:operator LAj}
Let $m\in \{1,\ldots, d\}$. Let $j\in \{1,\ldots, m\}$ and let $J=\{j_1,\ldots, j_m\}$ with $1\leq j_1<\ldots <j_m\leq d$. Let $B\in \mathcal M_\infty(m\times m)$. The {\sl $(j, J)$-determinantal operator associated to $B$} is the   operator $L_{B, j, J}:\Ha \to H^2(\D)$ defined as follows:
\[
L_{B, j, J}(f_1\e_1+\ldots+f_d\e_d):=\det \left(\begin{matrix} b_{11} & \ldots &  b_{1m}\\
\ldots & \ldots & \ldots \\  b_{(j-1)1} & \ldots &  b_{(j-1)m}\\ f_{j_1} & \ldots & f_{j_m}\\  b_{(j+1)1} & \ldots & b_{(j+1)m}\\ \ldots & \ldots & \ldots \\  b_{m1} & \ldots &  b_{mm}\\  \end{matrix} \right)
\]
In case $m=d$ (hence $J=\{1,\ldots, d\}$) we simply denote $L_{B, j, J}$ by $L_{B,j}$.
\end{definition}

\begin{lemma}\label{Lem:continuous-S-L}
Let $m\in \{1,\ldots, d\}$.  Let $B\in \mathcal M_\infty(m\times m)$. For every $j\in \{1,\ldots, m\}$ and  $J=\{j_1,\ldots, j_m\}$ with $1\leq j_1<\ldots <j_m\leq d$, the operator $L_{B, j, J}:\Ha \to H^2(\D)$ is continuous and $L_{B, j, J} \circ S= S_1 \circ L_{B, j, J}$.
\end{lemma}
\begin{proof}
Let $F_1, F_2\in \Ha$. Since the determinant is $m$-multilinear on the rows, it follows that $L_{B, j, J}(F_1+F_2)=L_{B, j, J}(F_1)+L_{B, j, J} (F_2)$ and $L_{B, j, J}(zF_1)=zL_{B, j, J}(F_1)$. Therefore the operator $L_{B, j, J}$ is linear and intertwines $S$ and $S_1$. Finally, using the Laplace expansion of the determinant with respect to the $j$-th row, we see that, if $F=f_1\e_1+\ldots +f_d\e_d\in \Ha$,
\begin{equation}\label{Eq:L-determ-Laplace}
L_{B, j, J}(F)=\sum_{k=1}^d (-1)^{k+j} f_{j_k} \det (B^{jk}),
\end{equation}
where $B^{jk}$ is the $(m-1)\times (m-1)$ matrix obtained from $B$ by removing the $j$-th row and the $k$-th column. Note that $\det (B^{jk})$ is a polynomial in the $b_{lm}$'s. Since $b_{lm}\in H^\infty(\D)$, it follows that there exists $C>0$ such that $\|\det (B^{jk})\|_\infty\leq C$ for all $k=1,\ldots, m$. Therefore,
\[
\|L_{B, j, J}(F)\|_{H^2(\D)}\leq C \sum_{k=1}^m \|f_{j_k}\|_{H^2(\D)}\leq C\sqrt{m}\sqrt{\sum_{k=1}^m \|f_{j_k}\|^2_{H^2(\D)}}\leq C\sqrt{m}\|F\|_\Ha,
\]
hence $L_{B,j, J}$ is bounded.
\end{proof}

A first consequence of this lemma, recalling the definition of determinantal subspaces from the introduction is the following:
\begin{corollary}\label{Cor:determinantal-space-S}
Every determinantal subspace of $\Ha$ is closed and $S$-invariant.
\end{corollary}

Now we go on with our construction:

\begin{definition}
Let $A\in\mathcal M_\infty(d\times d)$. Let $\mathcal A(A)=\{a_{jk}: a_{jk}\not\equiv 0\}$. If $A$ is not the zero matrix---hence $\mathcal A(A)\neq\emptyset$---we denote by $\phi_A$ the (inner function) greatest common divisor of $\{I(a_{jk}): a_{jk}\in\mathcal A(A)\}$. We call $\phi_A$  the {\sl inner greatest common divisor of $A$}.
\end{definition}

Finally for every $j,k\in \{1,\ldots, d\}$ let
\[
\hat a_{jk}:=\frac{a_{jk}}{\phi_A}.
\]
Note that $\hat a_{jk}\in H^\infty(\D)$ for all $j,k=1,\ldots, d$.

\begin{definition}\label{Def:reduc-matrix}
We call the matrix $\hat A:=(\hat a_{jk})_{j,k=1,\ldots, d}\in\mathcal M_\infty(d\times d)$ the {\sl reduced matrix} of $A$.
\end{definition}

\subsection{Case $\det A\not\equiv 0$}

Let $A\in\mathcal M_\infty(d\times d)$ and assume  $\det A\not\equiv 0$. Clearly  $\det \hat A\not\equiv 0$. Hence, it is well defined
\[
\varphi_A:=I(\det \hat A),
\]
 the inner factor of $\det \hat A$.

Recall that, if $L:\Ha\to H^2(\D)$ is a linear operator and $V\subseteq H^2(\D)$, the fiber of $L$ over $V$ is defined as
\[
L^{-1}(V):=\{F\in \Ha: L(F)\in V\}.
\]
Clearly, if $L$ is continuous and $V$ is closed in $H^2(\D)$, then $L^{-1}(V)$ is closed in $\Ha$.

\begin{definition}\label{Def:NA}
Let $A\in\mathcal M_\infty(d\times d)$, and assume $\det A\not\equiv 0$. We let
\[
\mathcal N_A:=\phi_A \left(\bigcap_{j=1}^d L_{\hat A, j}^{-1}(\varphi_A H^2(\D)) \right),
\]
and call it the {\sl closed $S$-invariant space associated to $A$}.
\end{definition}

With the notation introduced in the Introduction,
\[
\mathcal N_A=\phi_A\mathcal Q_{\varphi_A},
\]
 where $\mathcal Q_{\varphi_A}$ is a determinantal space.

\subsection{Case $\det A\equiv 0$} Let $A\in\mathcal M_\infty(d\times d)$ and assume  $\det A\equiv 0$.

 As a matter of notation, let $m,n$ be positive integers. If $z\in\D$, and $B\in \mathcal M_\infty(m\times n)$, we denote by $B(z)$ the $m\times n$ matrix whose entries are the complex numbers given by evaluating the entries of $B$ at $z$.

\begin{definition}\label{Def:rank-of-A}
Let $B\in\mathcal M_\infty(m\times n)$  where $m, n$ are positive intergers. We say that $B$ has {\sl rank} $k\in \{0,\ldots, \min\{m,n\}\}$ if $\hbox{rank}(B(z))\leq k$ for all $z\in\D$ and if there exists $z_0\in \D$ such that $\hbox{rank}(B(z_0))=k$.
\end{definition}

\begin{remark}\label{Rem:rank-k-means-det}
Note that if $A\in\mathcal M_\infty(d\times d)$ then $A$ has rank $d$ if and only if $\det A\not\equiv 0$. On the other hand, by Kronecker's theorem, the rank of $A$ is $k<d$ if and only if there exists a $k\times k$ matrix $B$ obtained by removing $d-k$ rows and $d-k$ columns from $A$    such that $\det B\not\equiv 0$, and all matrices containing $B$ and obtained  from $A$ by removing $d-k-1$ rows and $d-k-1$ columns  have determinants identically zero.
\end{remark}

\begin{remark}\label{Rem:same-rank-A-A-hat}
Clearly, $A$ is the zero matrix (that is the $d\times d$ matrix whose entries are the identically $0$ function) if and only if $\hbox{rank}(A)=0$. On the other hand, if $A$ is not the zero matrix, then $A=\phi_A\hat A$ and, since $\phi_A$ has at most a discrete set of zeros in $\D$, it is clear by the previous observation that $\hbox{rank}(\hat A)=\hbox{rank}(A)$.
\end{remark}

\begin{definition}
Let $A\in \mathcal M_\infty(d\times d)$, and assume that $\hbox{rank}(A)=k\in\{1,\ldots, d-1\}$. We let $\mathcal R(A)$ be the set of all $f_1\e_1+\ldots+f_d\e_d\in\Ha$ such that
\begin{equation}\label{Eq:rank-f-A}
 \hbox{rank}\left(\begin{matrix} f_1 & \ldots &f_d\\ \hat a_{11}& \ldots &\hat a_{1d}\\ \vdots& \ldots & \vdots\\  \hat a_{d1}& \ldots &\hat a_{dd} \end{matrix}  \right)= k.
\end{equation}
\end{definition}

\begin{lemma}\label{lem:r(A)-S-closed-inv}
Let $A\in\mathcal M_\infty(d\times d)$, and assume that $\hbox{rank}(A)=k\in\{1,\ldots, d-1\}$. Then $\mathcal R(A)$ is a closed $S$-invariant subspace of $\Ha$. Moreover, $\mathcal R(A)$ is a determinantal space of the form $\mathcal Q_0$.
\end{lemma}
\begin{proof}
Since the rank is invariant by switching rows or columns, we can assume that
\[
\det \left(\begin{matrix}\hat a_{11}& \ldots & \hat a_{1k}\\ \vdots& \ldots& \vdots\\ \hat a_{k1}  & \ldots & \hat a_{kk} \end{matrix}\right)\not\equiv 0.
\]
By Kronecker's theorem (see Remark~\ref{Rem:rank-k-means-det}) the condition \eqref{Eq:rank-f-A} is  equivalent to
\begin{equation}\label{Eq:rank-as-det}
R_m(f_1\e_1+\ldots+f_d\e_d):=\det\left(\begin{matrix} f_1 & \ldots &f_k & f_m\\ \hat a_{11}& \ldots & \hat a_{1k}& \hat a_{1m}\\ \vdots& \ldots& \vdots& \vdots \\ \hat a_{k1}  & \ldots & \hat a_{kk}& \hat a_{km} \end{matrix}  \right)\equiv 0, \quad m=k+1,\ldots d.
\end{equation}
By Lemma~\ref{Lem:continuous-S-L},  $R_m:\Ha\to H^2(\D)$ is a continuous linear operator and $R_m \circ S= S_1 \circ R_m$, $m=k+1,\ldots d$. Hence,
\[
\mathcal R(A)=\bigcap_{m=k+1}^d R_m^{-1}(\{0\}),
\]
is a closed, $S$-invariant subspace of $\Ha$, and it is actually a determinantal space of the form $\mathcal Q_0$.
\end{proof}

Recalling from Remark~\ref{Rem:same-rank-A-A-hat} that $A$ and $\hat A$ have the same rank, we give the following:

\begin{definition}
Let $A\in\mathcal M_\infty(d\times d)$ and assume $\det A\equiv 0$ and $\hbox{rank}(A)=k\in\{1,\ldots, d-1\}$. We say that $J=\{j_1, \ldots, j_k: 1\leq j_1<\ldots < j_k\leq d\}$ is a {\sl good multi-index} of $\hat A$ if there exists $S=\{s_1, \ldots, s_k: 1\leq s_1<\ldots < s_k\leq d\}$ such that the minor $\hat A^{S,J}\in\mathcal M_\infty(k\times k)$ of $\hat A$ obtained from $\hat A$ by removing the rows $s\not\in S$ and the columns $j\not\in J$ has the property that $\det \hat A^{S,J}\not\equiv 0$.

The set of all indices $S$ which satisfy the previous condition with respect to $J$ is denoted by $\mathcal J_J(\hat A)$.
\end{definition}

Note that, by Remark~\ref{Rem:rank-k-means-det}, if $A$ has rank $k$ there exists at least one good multi-index for $\hat A$.

\begin{example}
Let
\[
A=\left(\begin{matrix}
\phi & 0 & 0 \\ 0 & \varphi & \varphi \\ 0 & 1 & 1
\end{matrix}\right),
\]
where $\phi,\varphi$ are non-invertible inner functions. Clearly, $A=\hat A$ and $\hbox{rank}(A)=2$. The good multi-indices of $\hat A$ are $J_1=\{1,2\}$ and $J_2=\{1,3\}$. Moreover, $\mathcal J_{J_1}(\hat A)=\mathcal J_{J_2}(\hat A)=\{\{1,2\}, \{1,3\}\}$. Also,
\[
\hat A^{\{1,2\}, J_1}=\hat A^{\{1,2\}, J_2}=\left(\begin{matrix}
\phi & 0 \\  0 & \varphi
\end{matrix}\right), \quad \hat A^{\{1,3\}, J_1}=\hat A^{\{1,3\}, J_2}=\left(\begin{matrix}
\phi & 0 \\  0 & 1
\end{matrix}\right).
\]
\end{example}

\begin{definition}
 Let $A\in\mathcal M_\infty(d\times d)$ and assume $\det A\equiv 0$ and $\hbox{rank}(A)=k\in\{1,\ldots, d-1\}$. Let $J$ be a good multi-index of $\hat A$. We let $\varphi_{J,A}$ to be the (inner function) greatest common divisor of the inner factors of $\det \hat A^{S,J}$, when $S$ varies in $\mathcal J_J(\hat A)$. That is,
 \[
 \varphi_{J,A}:=\hbox{g.c.d}\{I(\det \hat A^{S,J}):  S\in\mathcal J_J(\hat A)\}.
 \]
\end{definition}

Now we are ready to define a closed $S$-invariant subspace associated to $A$ in case $\det A\equiv 0$:

\begin{definition}\label{Def:NA-det0}
Let $A\in\mathcal M_\infty(d\times d)$ and assume $\det A\equiv 0$ and $\hbox{rank}(A)=k\in\{1,\ldots, d-1\}$. Let $J$ be a good multi-index of $\hat A$. We let
\[
\mathcal N_{A,J}:=\phi_A\left(  \left(\bigcap_{S\in \mathcal J_J(\hat A), j=1,\ldots, k}L^{-1}_{\hat A^{S,J}, j, J} (\varphi_{J,A}H^2(\D))\right)\cap \mathcal R(A)\right).
\]
\end{definition}

It follows at once from Lemma~\ref{Lem:continuous-S-L} and Lemma~\ref{lem:r(A)-S-closed-inv} that $\mathcal N_{A,J}$ is a closed $S$-invariant subspace of $\Ha$ for every good multi-index $J$ of $\hat A$. Also, note that, in the terminology of the Introduction,
\[
\mathcal N_{A,J}=\phi_A (\mathcal Q_{\varphi_{J,A}}\cap \mathcal Q_0).
\]

\begin{remark}
Let $A\in\mathcal M_\infty(d\times d)$ and assume $\det A\equiv 0$ and $\hbox{rank}(A)=k\in\{1,\ldots, d-1\}$. Let $J, J'$ be a good multi-indices of $\hat A$. We will show (see the proof of Theorem~\ref{Thm:subspace-includL}) that actually $\mathcal N_{A,J}=\mathcal N_{A,J'}$. In fact, we show that, for every $J$ good  multi-index of $\hat A$, we have
\[
\mathcal N_{A,J}=\overline{\hbox{span}\{ S^n(a_{11}\e_1+\ldots+a_{1d}\e_d),\ldots, S^n(a_{d1}\e_1+\ldots+ a_{dd}\e_d): n\in \N\}}^\Ha.
\]
\end{remark}

\begin{example} Let
\[
A=\left(\begin{matrix} \phi &0 & 0\\ 0& \varphi & 0 \\ 0  &  0& 0 \end{matrix}  \right),
\]
with $\varphi, \phi$ two non-invertible inner functions, coprime. Hence, $A=\hat A$, $\hbox{rank}(A)=2$ and the (only) good multi-index is $J=\{1,2\}$. Also, $\mathcal J_J(\hat A)=\{J\}$. We have $\varphi_{J,A}=\phi\varphi$. Hence,
\begin{equation*}
\begin{split}
L^{-1}_{\hat A^{J,J}, 1, J} (\phi\varphi H^2(\D))&=\{f_1\e_1+f_2\e_2+f_3\e_3\in \Ha: f_1\varphi \in\phi\varphi H^2(\D)\}\\&=\{f_1\e_1+f_2\e_2+f_3\e_3\in \Ha: f_1 \in\phi H^2(\D)\},
\end{split}
\end{equation*}
and, similarly,
\begin{equation*}
L^{-1}_{\hat A^{J,J}, 2, J} (\phi\varphi H^2(\D))=\{f_1\e_1+f_2\e_2+f_3\e_3\in \Ha: f_2 \in\varphi H^2(\D)\}.
\end{equation*}
Also, by \eqref{Eq:rank-as-det},
\[
\mathcal R(A)=H^2(\D)\e_1\oplus H^2(\D)\e_2,
\]
so that
\[
\mathcal N_ A=\phi H^2(\D)\e_1\oplus \varphi H^2(\D)e_2.
\]
\end{example}

\begin{example} Let
\[
A=\left(\begin{matrix} 1 &0 & 0\\ 0& 0 & 0 \\ 0  &  0& 0 \end{matrix}  \right).
\]
 Hence, $A=\hat A$, $\hbox{rank}(A)=1$ and the (only) good multi-index is $J=\{1\}$. Also, $\mathcal J_J(\hat A)=\{J\}$. We have $\varphi_{J,A}=1$. Hence,
\begin{equation*}
L^{-1}_{\hat A^{J,J}, 1, J} (H^2(\D))=\{f_1\e_1+f_2\e_2+f_3\e_3\in \Ha: f_1\in H^2(\D)\}=\Ha.
\end{equation*}
Also, by \eqref{Eq:rank-as-det},
\[
\mathcal R(A)=H^2(\D)\e_1,
\]
so that
\[
\mathcal N_ A= H^2(\D)\e_1.
\]
\end{example}

\section{The spaces $\mathcal N_A$ and $\mathcal N_{J,A}$ are not maximal for the shift}

In this section we prove an analogue of Lemma~\ref{Lem:dim-one} for the spaces defined in the previous section.

\begin{proposition}\label{Prop:NA-no-maximal}
Let $A\in\mathcal M_\infty(d\times d)$, $\det A\equiv 0$. Suppose $A$ is not the identically zero matrix. Let $J$ be a good multi-index of $\hat A$. Then $\dim (\mathcal N_{A,J})^\perp=\infty$ and there exists a closed $S$-invariant subspace $M\subsetneq \Ha$ such that $\mathcal N_{A,J}\subsetneq M$.
\end{proposition}
\begin{proof}
Let $A=(a_{jm})$, $j,m\in \{1,\ldots, d\}$ and $a_{jm}\in H^\infty(\D)$.

Let $k\in\{1,\ldots, d-1\}$ be the rank of $A$--and hence of $\hat A$. In order to simplify notation, and without loss of generality, we can assume that $J=\{1,\ldots, k\}$ is a good multi-index of $\hat A$ and that $J\in \mathcal J_J(\hat A)$. Thus, $\det \hat A^{J,J}\not\equiv 0$.

Now, for every $\theta\in H^2(\D)$,
\begin{equation}\label{Eq:change-matrix-B}
\det\left(\begin{matrix} 0 & \ldots &0 & \theta\\ \hat a_{11}& \ldots & \hat a_{1k}& \hat a_{1d}\\ \vdots& \ldots& \vdots& \vdots \\ \hat a_{k1}  & \ldots & \hat a_{kk}& \hat a_{kd} \end{matrix}  \right)=(-1)^{k} \theta\det \hat A^{J,J}.
 \end{equation}
In particular, if $\theta\not\equiv 0$, by \eqref{Eq:rank-as-det},  $\theta\e_d\not\in \mathcal R(A)$ and,  hence $\mathcal R(A)\neq\Ha$. Note that this also implies that $\dim (\Ha/\mathcal R(A))=\infty$---hence, $\dim (\mathcal R(A))^\perp=\infty$ and hence $\dim (\mathcal N_{A,J})^\perp=\infty$---because $[\theta \e_d]=[0]$ in $\Ha/\mathcal R(A)$ if and only if $\theta\equiv 0$ and thus $H^2(\D)\ni \theta\mapsto [\theta \e_d]\in \Ha/\mathcal R(A)$ is a linear injective operator.

We know that $\mathcal N_{A,J}\subseteq \phi_A \mathcal R(A)$ and, clearly, $\phi_A \mathcal R(A)\subseteq  \mathcal R(A)$. Thus, if $\mathcal N_{A,J}\neq \mathcal R(A)$, we can take $M=\mathcal R(A)$ and we are done since $\mathcal R(A)$ is a closed $S$-invariant subspace of $\Ha$. Therefore, we can assume that
\[
\mathcal N_{A,J}=\mathcal R(A).
\]
Now, let  $\theta$ be an inner function and consider the matrix $B\in \mathcal M_\infty(d\times d)$ given by
\[
B:=\left(\begin{matrix} 0 & \ldots &0 & \theta\\ \hat a_{11}& \ldots & \hat a_{1, d-1}& \hat a_{1d}\\ \vdots& \ldots& \vdots& \vdots \\ \hat a_{d-1,1}  & \ldots & \hat a_{d-1, d-1}& \hat a_{d-1,d} \end{matrix}  \right).
\]
By \eqref{Eq:change-matrix-B}, the rank of $B$ is $k+1$.

\medskip

{\sl 1. Case  $k+1<d$.}

\medskip

We claim that $\mathcal R(A)\subsetneq \mathcal R(B)\subsetneq \Ha$. If this is so, taking into account that $\mathcal R(B)$ is closed and $S$-invariant,  we can take $M=\mathcal R(B)$ and we are done. Indeed, arguing as before, we see that $\mathcal R(B)\subsetneq \Ha$. On the other hand, let $f_1\e_1+\ldots+f_d\e_d\in\mathcal R(A)$. By \eqref{Eq:rank-as-det}, we have  for $m=k+1,\ldots, d-1$,
\[
\det\left(\begin{matrix} f_1& \ldots & f_k & f_m & f_d \\ 0 & \ldots &0 & 0& \theta\\ \hat a_{11}& \ldots & \hat a_{1k}& \hat a_{1m} & \hat a_{1d}\\ \vdots& \ldots& \vdots& \vdots \\ \hat a_{k1}  & \ldots & \hat a_{kk}& a_{km}& \hat a_{k,d} \end{matrix}  \right)=(-1)^{k}\theta\det \left(\begin{matrix} f_1& \ldots & f_k & f_m \\ \hat a_{11}& \ldots & \hat a_{1k}& \hat a_{1m} \\ \vdots& \ldots& \vdots& \vdots \\ \hat a_{k1}  & \ldots & \hat a_{kk}& a_{km}\end{matrix}  \right)\equiv 0,
\]
which, by \eqref{Eq:change-matrix-B} and Kronecker's Theorem (see Remark~\ref{Rem:rank-k-means-det}) implies that
\[
\hbox{rank}\left(\begin{matrix} f_1& \ldots  & f_{d-1} & f_d \\ 0 & \ldots & 0& \theta\\ \hat a_{11}& \ldots &  \hat a_{1, d-1} & \hat a_{1d}\\ \vdots& \ldots& \vdots& \vdots \\ \hat a_{d-1,1}  & \ldots &  a_{d-1, d-1}& \hat a_{d-1,d} \end{matrix}  \right)= k+1,
\]
that is, $f_1\e_1+\ldots+f_d\e_d\in\mathcal R(B)$. Hence, $\mathcal R(A)\subset \mathcal R(B)$. It is clear that $\mathcal R(A)\neq \mathcal R(B)$, since $\theta \e_d\in \mathcal R(B)$ but $\theta \e_d\not\in \mathcal R(A)$.

\medskip

{\sl 1. Case  $k+1=d$.}

\medskip

Since $k=d-1$, we have that $J=\{1,\ldots, d-1\}$ (and we are assuming $\det \hat A^{J,J}\not\equiv 0$). Let $\theta$ be a non-invertible inner function coprime   with the inner part $I(\det \hat A^{J,J})$ of $\det \hat A^{J,J}$ (for instance, if $I(\det \hat A^{J,J})(z_0)\neq 0$ and for some $z_0\in\D$, we can take $\theta(z)=\frac{z_0-z}{1-\overline{z_0}z}$, $z\in \D$). Let
\[
M:=\bigcap_{j=1}^d L_{B, j}^{-1}(\theta H^2(\D)).
\]
Note that $M$ is a determinantal space, so it is closed and $S$-invariant. We have
\[
L_{B, 1}(\e_d)=\det\left(\begin{matrix} 0 & \ldots &0 & 1\\ \hat a_{11}& \ldots & \hat a_{1, d-1}& \hat a_{1d}\\ \vdots& \ldots& \vdots& \vdots \\ \hat a_{d-1,1}  & \ldots & \hat a_{d-1, d-1}& \hat a_{d-1,d} \end{matrix}  \right)=(-1)^{d+1}\det \hat A^{J,J}.
\]
Taking into account that  $\theta$ is coprime with $I(\det \hat A^{J,J})$,  it follows that $L_{B, 1}(\e_d)\not\in \theta H^2(\D)$. Thus, $M\neq \Ha$.

Clearly, $\theta \e_d\in M$, since $L_{B,j}(\theta\e_d)\equiv 0$ for $j=2,\ldots, d$ and
\[
L_{B,1}(\theta \e_d)=(-1)^{d+1}\theta\det \hat A^{J,J}\in \theta H^2(\D).
\]
While, as we already noticed, $\theta\e_d\not\in \mathcal R(A)$. Therefore, $M\neq \mathcal R(A)$. So, in order to complete the proof, we are left to show that
\[
\mathcal R(A)\subset M.
\]
Let $f_1\e_1+\ldots+f_d\e_d\in\mathcal R(A)$. Hence, by Kronecker's Theorem (see Remark~\ref{Rem:rank-k-means-det}), $L_{B,1}(f_1\e_1+\ldots+f_d\e_d)=0$. While, for $j\in\{2,\ldots, d\}$ we have
\[
L_{B,j}(f_1\e_1+\ldots+f_d\e_d)=(-1)^{d+1}\theta\det\left(\begin{matrix}\hat a_{11} & \ldots &\hat a_{1,d-1}\\ \vdots& \vdots& \vdots \\ \hat a_{j-1,1} & \ldots &\hat a_{j-1,d-1}\\ f_1 & \ldots & f_{d-1}\\ \hat a_{j+1,1} & \ldots &\hat a_{j+1,d-1}\\  \vdots& \vdots& \vdots \\ \hat a_{d-1,1} & \ldots &\hat a_{d-1,d-1} \end{matrix} \right)\in\theta H^2(\D).
\]
Therefore, $f_1\e_1+\ldots+f_d\e_d\in M$, and we are done.
\end{proof}

\begin{proposition}\label{Prop:NA-no-maximal2}
Let $A\in\mathcal M_\infty(d\times d)$, $\det A\not\equiv 0$ and suppose that $\dim (\mathcal N_A)^\perp\geq 2$. Then there exists a closed $S$-invariant subspace $M\subsetneq \Ha$ such that $\mathcal N_A\subsetneq M$.
\end{proposition}

\begin{proof}
Let
\[
\mathcal Q_{\varphi_A}:=\bigcap_{j=1}^d L_{\hat A, j}^{-1}(\varphi_A H^2(\D)).
\]
 By the very definition, $\mathcal N_A:=\phi_A \mathcal Q_{\varphi_A}$.

 Suppose first that $\mathcal Q_{\varphi_A}=\Ha$. Hence, $\mathcal N_A=\phi_A \Ha$. Since we are assuming that $\dim (\mathcal N_A)^\perp\geq 2$, it follows that $\phi_A$ is not invertible. If $d=1$, the result follows from Lemma~\ref{Lem:dim-one}, so we can assume $d\geq 2$. In this case, let
 \[
 M:=\left(\phi_A H^2(\D)\right)\e_1\oplus H^2(\D)\oplus\ldots\oplus H^2(\D).
 \]
Clearly $M$ is closed, $S$-invariant, $\phi_A \Ha\subsetneq M\subsetneq\Ha$, and we are done.

Therefore, we can assume that $\mathcal Q_{\varphi_A}\neq\Ha$. Hence, there exists $j$ such that $L_{\hat A, j}^{-1}(\varphi_A H^2(\D))\neq \Ha$. Without loss of generality, we can assume $j=d$. Since $\mathcal N_A\subseteq \phi_A L_{\hat A, d}^{-1}(\varphi_A H^2(\D))$ and clearly $\phi_A L_{\hat A, d}^{-1}(\varphi_A H^2(\D) \subseteq L_{\hat A, d}^{-1}(\varphi_A H^2(\D))$, if $\mathcal N_A\neq L_{\hat A, d}^{-1}(\varphi_A H^2(\D))$, we can take $M=L_{\hat A, d}^{-1}(\varphi_A H^2(\D))$. By Corollary~\ref{Cor:determinantal-space-S}, $M$ is closed and $S$-invariant, and we are done.

Therefore, we are left to assume
\[
\mathcal N_A=L_{\hat A, d}^{-1}(\varphi_A H^2(\D)).
\]

Fix $j\in \{1,\ldots, d\}$ and let $S_j=\{1,\ldots, d\}\setminus\{j\}$. Let $J=\{1,\ldots, d-1\}$.  Also, let $\tilde\theta_j$ be the (inner function) greatest common divisor of $\varphi_A$ and the inner part $I(\det \hat A^{J,S_j})$ of $\det \hat A^{J,S_j}$. Let  $\theta_j$ be an inner function (unique up to multiplication by a unimodular constant) such that
\begin{equation}\label{Eq:varphi-divisors}
\varphi_A=\tilde{\theta_j}\theta_j.
\end{equation}
 Hence
\begin{equation}\label{Eq:e-j-in-L}
\begin{split}
L_{\hat A, d}(\theta_j\e_j)&=\det\left(\begin{matrix}\hat a_{11} &\ldots & \hat a_{1,j-1} &  \hat a_{1,j} & \hat a_{1,j+1}  &\ldots & \hat a_{1,d}\\   \vdots &\vdots  & \vdots  &  \vdots \\ \hat a_{d-1,1} &\ldots & \hat a_{d-1,j-1} & \hat a_{d-1,j} & \hat a_{d-1,j+1} & \ldots & \hat a_{d-1,d}\\ 0 & \ldots & 0&\theta_j & 0&\ldots& 0\end{matrix}\right)\\&=(-1)^{d+j}\theta_j \det \hat A^{J,S_j}\in \varphi_AH^2(\D).
\end{split}
\end{equation}
This in particular implies that
\begin{equation}\label{Z-inside}
Z:=\theta_1 H^2(\D)\oplus\ldots \oplus \theta_d H^2(\D)\subseteq L_{\hat A, d}^{-1}(\varphi_A H^2(\D)).
\end{equation}

\medskip
{\sl Case 1.} There exists $j\in\{1,\ldots, m\}$ such that $\theta_j$ is neither invertible nor a Blaschke product with a simple pole.
\medskip

We can assume $j=1$, the other cases being similar. In this case, by Lemma~\ref{Lem:dim-one}, there exists a not invertible inner function $\varphi_1$ such that $\varphi_1$ divides $\theta_1$ and $\theta_1$ does not divide $\varphi_1$. Let
\[
M:=L_{\hat A, d}^{-1}(\varphi_1 H^2(\D)).
\]
Note that $M$ is closed and $S$-invariant. Since $\varphi_1$ divides $\theta_1$, and hence, by \eqref{Eq:varphi-divisors}, it divides $\varphi_A$, it follows easily that $L_{\hat A, d}^{-1}(\varphi_A H^2(\D))\subset M$. Moreover, since for every $h\in H^2(\D)$,
\[
L_{\hat A, d}(h\e_1)=(-1)^{d+1}h \det \hat A^{J,S_1},
\]
and since $\varphi_1$ and $I(\det \hat A^{J,S_1})$ are coprime, it follows that $\varphi_1\e_1\in M$ but $\varphi_1\e_1\not\in L_{\hat A, d}^{-1}(\varphi_A H^2(\D))$, and that $\e_1\not\in M$. That is, $M\neq \Ha$ and $M$ properly contains $L_{\hat A, d}^{-1}(\varphi_A H^2(\D))$.

\medskip
{\sl Case 2.} For all $j\in\{1,\ldots, m\}$ either $\theta_j$ is  a Blaschke product with a simple pole or $\theta_j$ is invertible.
\medskip

Since $L_{\hat A, d}^{-1}(\varphi_A H^2(\D))\neq\Ha$, there exists at least one $j$ such that $\theta_j$ is not invertible. To simplify readability, we  assume that $\theta_1,\ldots, \theta_m$ are Blaschke products with a simple pole and $\theta_{m+1},\ldots, \theta_d$ are invertible, for some $1\leq m\leq d$ (the other cases are similar).

Let $Q:=L_{\hat A, d}^{-1}(\varphi_A H^2(\D))$. By \eqref{Z-inside}, $Z\subseteq Q$. Since $Z$ is an orthogonal direct sum, it is easy to see that there is a natural isometrical isomorphim between $\Ha/Z$ and
\[
E:=\left(H^2(\D)/(\theta_1H^2(\D))\right)\e_1+\ldots+\left(H^2(\D)/(\theta_m H^2(\D))\right)\e_m,
\]
given by
\[
E\ni [f_1]\e_1+\ldots+[f_m]\e_m\mapsto [f_1\e_1+\ldots+f_m\e_m]\in \Ha/Z.
\]
 Since $\theta_j$ is a Blaschke product with a simple pole  it follows that $H^2(\D)/(\theta_jH^2(\D))$ is a one-dimensional space for $j=1,\ldots, m$,. Hence, $\dim \Ha/Z=\dim E=m$. Taking into account that $\Ha/Z$ is isometrically isomorphic to $Z^\perp$, it follows that $\dim Z^\perp=m$.

From $Z\subseteq Q$, we have that $Q^\perp\subseteq Z^\perp$, hence, $\dim Q^\perp\leq m$. Note that $Q^\perp\neq\{0\}$ (because $Q\neq \Ha$) and $Q^\perp$ is $S^\ast$-invariant (because $Q$ is $S$-invariant).

Therefore, $S^\ast|_{Q^\perp}:Q^\perp\to Q^\perp$ is a linear endomorphism of a finite dimensional space and $2\leq \dim Q^\perp\leq m$. Hence,  there exists a subspace $V\neq\{0\}$ such that $V\subsetneq Q^\perp$ and $S^\ast(V)\subseteq V$. Given such a $V$, we let $M=V^\perp$. Thus, $M$ is a closed $S$-invariant subspace of $\Ha$, $Q\subsetneq M$ and $M\neq\Ha$.
\end{proof}

\section{The Beurling-Lax matrix of an invariant subspace}\label{Sec:Beurling-Lax}

Let $M\subset \Ha$ be a closed $S$-invariant subspace. By the Beurling-Lax Theorem \cite{Beu, Lax} (see also, {\sl e.g.} \cite[Section~1.12]{RoRo}), there exists a bounded linear operator $Q:\Ha \to \Ha$ (called an {\sl $S$-inner operator}) such that:
\begin{enumerate}
\item  there is a closed $S$-invariant orthogonal decomposition $\Ha=I_Q\stackrel{\perp}{\oplus} \ker Q$,
\item $A|_{I_Q}: I_Q\to Q(I_Q)=Q(\Ha)=M$ is an isometry,
\item $Q \circ S= S\circ Q$.
\end{enumerate}

The following proposition can be deduced also from \cite[Theorem~B Section~1.15]{RoRo}, but we give a proof both for the sake of completeness and in order to fix some notations.

\begin{proposition}\label{Prop:structure-invariant-shif-space}
Let $M\subset \Ha$ be a closed $S$-invariant subspace. Then there exist $a_{jk}\in  H^\infty(\D)$,  with $\|a_{jk}\|_\infty\leq 1$, $j,k=1,\ldots, k$,  such that
\begin{equation*}
\begin{split}
M&=\left\{\sum_{j=1}^d h_j (a_{j1}\e_1+\ldots +a_{jd}\e_d) : h_j\in H^2(\D) \right\}\\&=\overline{\hbox{span}\{S^n(a_{11}\e_1+\ldots+a_{1d}\e_d),\ldots, S^n(a_{d1}\e_1+\ldots+a_{dd}\e_d): n\in \N\}}^\Ha.
\end{split}
\end{equation*}
\end{proposition}
\begin{proof}
Let $Q$ be a $S$-inner operator such that $M=Q(\Ha)$. Let $k=1,\ldots, d$. Since $Q\circ S=S\circ Q$, it turns out that for every polynomial $q$,
\[
Q(q \e_k)=q Q(\e_k).
\]
If $f\in H^2(\D)$ and $\{q_j\}$ is a sequence of polynomials converging to $f$ in $H^2(\D)$, since $Q$ is a $S$-inner operator, we have
\begin{equation*}
\begin{split}
\lim_{j\to \infty} \|q_j Q( \e_k)-Q(f\e_k)\|_{\Ha}&= \lim_{j\to \infty}\| Q( q_j \e_k- f\e_k)\|_{\Ha}\\& \leq \lim_{j\to \infty}\|q_j \e_k- f\e_k\|_{\Ha}=\lim_{j\to \infty}\|q_j - f\|_{H^2(\D)}=0.
\end{split}
\end{equation*}
Therefore $\{q_j Q(\e_k)\}$ converges to $Q(f\e_k)$ in $\Ha$.

Now, as a matter of notation, if $F\in \Ha$, we  write $F=p_1(F)\e_1+\ldots+p_d(F)\e_d$. Note that, if $g\in H^2(\D)$ and $gF\in \Ha$ then $p_m(gF)=gp_m(F)$, $m=1,\ldots, d$.

Since $\{q_j Q(\e_k)\}$ converges to $Q(f\e_k)$ in $\Ha$ it follows that $\{p_m(q_j Q(\e_k))\}$ converges in $H^2(\D)$ to $p_m(Q(f\e_k))$, $m=1,\ldots, d$. Therefore,  $\{p_m(q_j Q( \e_k))\}$ converges uniformly on compacta to $p_m(Q(f\e_k))$.

Now we show that $\{p_m(q_j Q( \e_k))\}$ converges uniformly on compacta of $\D$ to $p_m(fQ(\e_k))$, $m=1,\ldots, d$, from which it follows that  $Q(f\e_k)=fQ(\e_k)$. Indeed, let $K\subset\subset \D$. Let $C_m:=\max_{\zeta \in K}|p_m(Q(\e_k))(\zeta)|$, $m=1,\ldots, d$. Since $\{q_j\}$ converges uniformly to $f$ on $K$, we have
\begin{equation*}
\begin{split}
&\max_{z\in K}|p_m(q_j(z)Q(\e_k)(z))-p_m(f(z)Q(\e_k)(z))|\\&=\max_{z\in K}|q_j(z) p_m(Q(\e_k)(z))-f(z) p_m(Q(\e_k)(z))| \leq C_m \max_{z\in K}|q_j(z) -f(z)|\to 0.
\end{split}
\end{equation*}
Therefore,
\begin{equation}\label{Eq:switchSA}
Q(f\e_k)=f Q(\e_k) \quad \forall f\in H^2(\D), k=1,\ldots, d.
\end{equation}
Now, for $j=1,\ldots, d$, we write
\[
Q(\e_j)=a_{j1}\e_1+\ldots+a_{jd}\e_d.
\]
By \eqref{Eq:switchSA}, for all $f\in H^2(\D)$,
\[
p_1(Q(f\e_1))=f a_{11}.
\]
Hence, taking into account that $Q$ is $S$-inner we have  for every $f\in H^2(\D)$,
\begin{equation}\label{eq:crea-multi-op}
\begin{split}
\|f a_{11}\|^2_{H^2(\D)}&\leq \|f a_{11}\|^2_{H^2(\D)}+\ldots+\|f a_{1d}\|^2_{H^2(\D)}= \|Q(f\e_1)\|^2_{\Ha}\\&\leq \|f \e_1\|^2_{\Ha}=\|f \|^2_{H^2(\D)}.
\end{split}
\end{equation}
Therefore, the multiplication operator $H^2(\D)\ni f\mapsto a_{11} f\in H^2(\D)$ is continuous and with (operator) norm $\leq 1$. By a standard result, $a_{11}\in H^\infty(\D)$ and $\|a_{11}\|_\infty\leq 1$. A similar argument works for $a_{jk}$, $j,k\in\{1,\ldots, d\}$.

Now,  every $f\in M$ is given by $f=Q(h)$ for some $h\in \Ha$ (actually, $h\in I_A$). If we write $h=h_1\e_1+\ldots+h_d\e_d$, we have
\[
f=Q(h)=h_1 Q(\e_1)+\ldots+h_d A(\e_d)=\sum_{j=1}^d h_j (a_{j1}\e_1+\ldots a_{jd}\e_d),
\]
and this proves the first part of the formula. As for the second, by Beurling theorem,
\[
\Ha=\overline{\hbox{span}\{S^n(\e_1),\ldots, S^n(\e_d): n\in \N\}}^\Ha.
\]
Since $Q \circ S= S\circ Q$, we have
\[
Q(\Ha)=\overline{\hbox{span}\{S^n(Q(\e_1)),\ldots, S^n(Q(\e_d)): n\in \N\}}^\Ha,
\]
and we are done.
\end{proof}

There are a couple of interesting corollaries we need. Before that, we need a definition.
Let
\[
\Ha_\infty:=\{f_1\e_1+\ldots+f_d\e_d: f_j\in H^\infty(\D), j=1,\ldots, d\}.
\]

\begin{corollary}\label{Cor:conten-chiuso-S}
Let $M, N$ be two closed $S$-invariant subspaces of $\Ha$. Then $M\subseteq N$ if and only if $M\cap \Ha_\infty\subseteq N$.
\end{corollary}
\begin{proof}
Assume first that $M\cap \Ha_\infty\subseteq N$. Let $\{a_{jk}\}_{j,k=1,\ldots, d}$ be the bounded functions given by Proposition~\ref{Prop:structure-invariant-shif-space} for $M$. Hence  $a_{j1}\e_1+\ldots+a_{jd}\e_d \in N$ for $j=1,\ldots, d$. Since $N$ is $S$-invariant and by  Proposition~\ref{Prop:structure-invariant-shif-space}, we have
\[
M=\overline{\hbox{span}\{S^n(a_{11}\e_1+\ldots+a_{1d}\e_d),\ldots, S^n(a_{d1}\e_1+\ldots+a_{dd}\e_d): n\in \N\}}^\Ha\subseteq N.
\]
The other implication of the corollary is trivial.
\end{proof}

\begin{proposition}\label{Prop:O-closed-S}
Let $M$ be a $S$-invariant subspace of $\Ha$. Let $O\in H^\infty(\D)$ be outer. Then
\[
\overline{OM}^\Ha=M.
\]
\end{proposition}
\begin{proof}
First we show that $\overline{OM}^\Ha\subseteq M$. Since $M$ is closed, it is enough to show that $OM\subseteq M$. Let $F\in M$. By Proposition~\ref{Prop:structure-invariant-shif-space}, there exist sequences of polynomials $\{p_{n,j}\}_{n\in\N}$, $j=1,\ldots, d$ such that $F$ is the limit in $\Ha$ of
the sequence $\{Q_n:=\sum_{j=1}^dp_{n,j}(a_{j1}\e_1+\ldots+a_{jd}\e_j)\}$. Now, again by Proposition~\ref{Prop:structure-invariant-shif-space},
\[
OQ_n=O\left(\sum_{j=1}^dp_{n,j}(a_{j1}\e_1+\ldots+a_{jd}\e_j)\right)=\sum_{j=1}^dp_{n,j}\left(O(a_{j1}\e_1+\ldots+a_{jd}\e_j)\right)\in M,
\]
therefore $OQ_n\in M$ for every $n$. Since $\|OQ_n-OF\|_\Ha\leq \|O\|_\infty \|Q_n-F\|_\Ha$, it follows that $\{OQ_n\}\subset M$ converges to $OF$ in $\Ha$, hence $OF\in M$, and we are done.

Now we prove that $M\subseteq \overline{OM}^\Ha$. Since the multiplication by $O$ commutes with $S$, we have that $OM$ is $S$-invariant, and so is $\overline{OM}^\Ha$. Therefore, by Corollary~\ref{Cor:conten-chiuso-S}  it is enough to show that $M\cap \Ha_\infty \subseteq \overline{OM}^\Ha$. To this aim, let $F\in M\cap \Ha_\infty$. Since $O$ is outer, hence a cyclic vector for $S_1$, there exists  a sequence of polynomials $\{p_n\}$ such that $\{p_nO\}$ converges to $1$ in $H^2(\D)$. Since $(p_nO)F=O(p_n F)$ and $p_nF\in M$, it follows that $p_nOF\in OM$ for all $n$. Now, write $F=f_1\e_1+\ldots+f_d\e_d$, with $f_j\in H^\infty(\D)$, $j=1,\ldots, d$. Let $C>0$ be such that $\|f_j\|_\infty<C$ for $j=1,\ldots, d$. Hence,
\begin{equation*}
\begin{split}
\|p_nOF-F\|^2_\Ha&=\sum_{j=1}^d\|(p_nO-1)f_j\|^2_{H^2(\D)}\leq \sum_{j=1}^d\|f_j\|^2_\infty\|p_nO-1\|^2_{H^2(\D)}\\ &\leq dC^2 \|p_nO-1\|^2_{H^2(\D)}.
\end{split}
\end{equation*}
Since $\lim_{n\to\infty}\|p_nO-1\|^2_{H^2(\D)}=0$, it follows that $\{p_nOF\}$ converges to $F$ in $\Ha$. But $\{p_nOF\}\subset OM$, hence $F\in \overline{OM}^\Ha$, and we are done.
\end{proof}

\begin{definition}\label{Def:BL-matrix}
Let $M\subset \Ha$ be a closed $S$-invariant subspace, and let $\{a_{jk}\}_{j,k=1,\ldots, d}$ be given by Proposition~\ref{Prop:structure-invariant-shif-space}. The matrix $A_M=(a_{jk})\in\mathcal M_\infty$ is a {\sl Beurling-Lax matrix} of $M$.
\end{definition}

\section{The structure theorem}

Let $N\subset \Ha$ be a closed $S$-invariant subspace. Let $A$ be a Beurling-Lax matrix associated to $N$ (see Definition~\ref{Def:BL-matrix}). Note that $A$ is the zero matrix if and only if $N=\{0\}$.

In case $A$ is not the zero matrix, and $\det A\not\equiv 0$, we can associate to $A$ the closed $S$-invariant subspace $\mathcal N_A$ (see Definition~\ref{Def:NA}). In case $\det A\equiv 0$ we can associate the closed $S$-invariant subspace $\mathcal N_{A,J}$  for every good multi-index of $\hat A$ (see Definition~\ref{Def:NA-det0}).

The main result of this section (from which Theorem~\ref{Thm:main-intro} follows at once) is the following:

\begin{theorem}\label{Thm:subspace-includL}
Let $\{0\}\subsetneq N\subsetneq\Ha$ be a closed $S$-invariant subspace. Let  $A$ be a Beurling-Lax matrix associated to $N$. Then,
\begin{itemize}
\item if $\det A\not\equiv 0$,
\begin{equation}\label{N-iswhat}
N= \mathcal N_A.
\end{equation}
\item While, if $\det A\equiv 0$,
\begin{equation}\label{N-iswhat0}
N= \mathcal N_{A, J},
\end{equation}
for every  good multi-index $J$ of $\hat A$. In particular, $N_{A, J}=N_{A,J'}$ for every $J, J'$ good multi-indices of $\hat A$.
\end{itemize}
\end{theorem}
\begin{proof}
Since $N\neq\{0\}$, $A$ is not the matrix identically zero. Let $\hat A$ be the reduced matrix of $A$ (see Definition~\ref{Def:reduc-matrix}). Hence, by Proposition~\ref{Prop:structure-invariant-shif-space}
\begin{equation*}
\begin{split}
N&=\overline{\hbox{span}\{S^n(a_{11}\e_1+\ldots+a_{1d}\e_d),\ldots, S^n(a_{d1}\e_1+\ldots+a_{dd}\e_d): n\in \N\}}^\Ha\\&
=\overline{\hbox{span}\{\phi_A S^n(\hat a_{11}\e_1+\ldots+\hat a_{1d}\e_d),\ldots, \phi_A S^n(\hat a_{d1}\e_1+\ldots+\hat a_{dd}\e_d): n\in \N\}}^\Ha.
\end{split}
\end{equation*}
Let
\[
W:=\overline{\hbox{span}\{S^n(\hat a_{11}\e_1+\ldots+\hat a_{1d}\e_d),\ldots, S^n(\hat a_{d1}\e_1+\ldots+\hat a_{dd}\e_d): n\in \N\}}^\Ha.
\]
Bering in mind that the multiplication by $\phi_A$ is an isometry in $\Ha$, we have
\[
N=\phi_A W.
\]

\medskip
\noindent {\sl Case A: $\det A\not\equiv 0$}.
\medskip

Fix $j\in\{1,\ldots, d\}$. Let $L_{A,j}$ be the determinantal operator associated to $A$ (see Definition~\ref{Def:operator LAj}). Since $S_1 \circ L_{A, j} = L_{A, j}\circ S$, it follows that, for $k\in\{1,\ldots, d\}\setminus\{j\}$,
\[
L_{A, j}\left(\hbox{span}\{S^n(\hat a_{k1}\e_1+\ldots+\hat a_{kd}\e_d) : n\in\N\} \right)=\{0\},
\]
hence, since $L_{A, j}$ is continuous,
\begin{equation*}
L_{A,j}(W)=L_{A,j}\left(\overline{\hbox{span}\{ S^n(\hat a_{j1}\e_1+\ldots+\hat a_{jd}\e_d): n\in \N\}}^\Ha\right).
\end{equation*}
Now, for every $n\in \N$ we have
\[
L_{A_j}(S^n(\hat a_{j1}\e_1+\ldots+\hat a_{jd}\e_d))=S_1^n(\det \hat A).
\]
Recall that,   by definition (and since $\det\hat A\not\equiv 0$)  $\varphi_A$ is the inner part of $\det \hat A$. We let $O$ to be the outer part of $\det \hat A$. Therefore, by Beurling theorem,
\begin{equation*}
\begin{split}
\overline{L_{A_j}(W)}^{H^2(\D)}&=
\overline{L_{A,j}\left(\overline{\hbox{span}\{ S^n(\varphi_A O): n\in \N\}}^\Ha\right)}^{H^2(\D)}\\&=\overline{\{\hbox{span}\{S_1^n(\varphi_A O): n\in \N\}}^{H^2(\D)}=\varphi_A H^2(\D).
\end{split}
\end{equation*}
From this it follows that $W\subseteq L_{A,j}^{-1}(\varphi_A H^2(\D))$ and then $N\subseteq \phi_A L_{A,j}^{-1}(\varphi_A H^2(\D))$, $j=1\ldots, d$, that is, $N\subseteq \mathcal N_A$.

Now, we show that $\mathcal N_A\subseteq N$. Since $N=\phi_A W$ and $\mathcal N_A=\phi_A \bigcap_{j=1}^d L_{A,j}^{-1}(\varphi_A H^2(\D))$ , it is enough to show that
\[
Z:=\bigcap_{j=1}^d L_{A,j}^{-1}(\varphi_A H^2(\D))\subseteq W.
\]
To this aim, let $F=f_1\e_1+\ldots f_d\e_d\in \Ha$. By \eqref{Eq:L-determ-Laplace}, $F\in Z$ if and only if
\[
L_{A, j}(F)=\sum_{k=1}^d (-1)^{k+j} f_k \det (\hat A^{jk})\in \varphi_A H^2(\D), \quad j=1,\ldots d.
\]
Since by Beurling theorem the closure of $\hbox{span}\{S_1^n(\varphi_A O): n\in \N\}$  is $\varphi_A H^2(\D)$, there exist  sequences of polynomials $\{p_{nj}\}_{n\in\N}$, $j=1,\ldots, d$, such that
\begin{equation}\label{Eq:determ-poly-Beur}
\lim_{n\to\infty}\|\varphi_A O p_{nj}-\sum_{k=1}^d (-1)^{k+j} f_k \det (\hat A^{jk})\|_{H^2(\D)}=0, \quad j=1,\ldots, d.
\end{equation}
Taking into account that $\hat a_{jk}\in H^\infty(\D)$, $j,k\in \{1,\ldots, d\}$---and hence if a sequence $\{h_n\}\subset H^2(\D)$ converges to $h$ in $H^2(\D)$ then $\{\hat a_{jk}h_n\}$ converges to $\hat a_{jk}h$ in $H^2(\D)$---from \eqref{Eq:determ-poly-Beur} we have for every $m=1,\ldots, d$
\begin{equation}\label{Eq:determ-poly-Beur2}
\lim_{n\to\infty}\|\sum_{j=1}^d \hat a_{jm} \varphi_A O p_{nj}-\sum_{k=1}^d  f_k \left(\sum_{j=1}^d  (-1)^{k+j} \hat a_{jm}\det (\hat A^{jk})\right)\|_{H^2(\D)}=0.
\end{equation}
Now we claim that
\begin{equation}\label{Eq:claim-det-jm}
\sum_{j=1}^d  (-1)^{k+j} \hat a_{jm}\det (\hat A^{jk})=\begin{cases}
\varphi_A O & k=m\\
0 & k\neq m
\end{cases}
\end{equation}
Assuming the claim for the moment, from \eqref{Eq:determ-poly-Beur2} we have
\[
\lim_{n\to\infty}\|\varphi_A \left(O\sum_{j=1}^d   p_{nj} \hat a_{jk} -O f_m\right)\|_{H^2(\D)}=0, \quad m=1,\ldots, d.
\]
Since $\varphi_A$ is inner, this implies that $\{ O\sum_{j=1}^d  p_{nj} \hat a_{jm}\}$ converges to $Of_m$ in $H^2(\D)$, $m=1,\ldots, d$. In other words, if we let
\[
Q_n:=p_{n1} (\hat a_{11}\e_1+\ldots+\hat a_{1d}\e_d)+\ldots+ p_{nd}(\hat a_{d1}\e_1+\ldots+\hat a_{dd}\e_d ),
\]
then $\{OQ_n\}$ converges  to $O(f_1\e_1+\ldots+f_d\e_d)=OF$ in $\Ha$. Note that $Q_n\in W$, so that $OQ_n\in OW$ for all $n$. By Proposition~\ref{Prop:O-closed-S}, $OW\subseteq W$, hence, $OQ_n\in W$ for all $n$. It follows that $OF\in W$. By the arbitrariness of $F$, this means that $OZ\subset W$ and, since $W$ is closed, $\overline{OZ}^\Ha\subseteq W$. Since $Z$ is $S$-invariant, by Proposition~\ref{Prop:O-closed-S} we have $Z=\overline{OZ}^\Ha$, hence $Z\subset W$ as wanted.

We are thus left to prove \eqref{Eq:claim-det-jm}.  Let
\[
U:=\{z\in \D: \det \hat A(z)\neq 0\}.
\]
Since $\det A\not\equiv 0$---hence $\det \hat A\not\equiv 0$---the set $U$ is open, connected and dense in $\D$. Let  $B$ the $d\times d$ matrix whose $(j,k)$ entry is $(-1)^{j+k}\det \hat A^{kj}$.
Then it follows from classical linear algebra that, for all $z\in U$
\[
\hat A(z)\cdot B(z)=(\det \hat A(z)) {\sf Id},
\]
where ${\sf Id}$ is the $d\times d$ identity matrix. Expanding the previous equality one obtains \eqref{Eq:claim-det-jm} for all $z\in U$. By analytic extension, it holds thus in $\D$ and we are done.

\medskip
\noindent {\sl Case B: $\det A\equiv 0$}.
\medskip

Let
\[
T(A):= \left(\bigcap_{S\in \mathcal J_J(\hat A), j=1,\ldots, k}L^{-1}_{\hat A^{S,J}, j, J} (\varphi_{J,A}H^2(\D))\right)\cap \mathcal R(A).
\]
Since
\[
\mathcal N_{A,J}=\phi_A T(A)
\]
 (see Definition~\ref{Def:NA-det0}) and $N=\phi_A W$, it is enough to show that $W=T(A)$.

To simplify the readability, we can assume that $J=\{1,\ldots, k\}$. This would not affect the proof since, switching the columns $j$ and $m$ in $\hat A$ corresponds to see the problem under the  isometric automorphism of $\Ha$  which switches the coefficients of $\e_j$ and $\e_m$.

Now, we claim that $W\subseteq T(A)$. To this aim, because of the (multi-)linearity of the determinant, it is enough to show  that $\underline{\hat a}_m:=\hat a_{m1} \e_1+\ldots +\hat a_{md} \e_d\in T(A)$ for $m=1,\ldots, d$. So, fix $m\in \{1,\ldots, d\}$. Obviously, $\underline{\hat a}_m\in \mathcal R(A)$. Next, let $S=\{1\leq s_1<\ldots<s_k\leq d\}\in\mathcal J_J(\hat A)$ and $j\in\{1,\ldots, k\}$.  We need to show that $L_{\hat A^{S,J}, j, J}(\underline{\hat a}_m) \in \varphi_{J,A}H^2(\D)$. Now,
\[
L_{\hat A^{S,J}, j, J}(\underline{\hat a}_m)=\det \left(\begin{matrix} \hat a_{s_11} & \ldots &  \hat a_{s_1k}\\ \ldots & \ldots & \ldots \\  \hat a_{s_{j-1}1} & \ldots &  \hat a_{s_{j-1}k}\\ \hat a_{m1} & \ldots & \hat a_{mk}\\  \hat a_{s_{j+1}1} & \ldots & \hat a_{s_{j+1}k}\\ \ldots & \ldots & \ldots \\  \hat a_{s_k1} & \ldots &  \hat a_{s_k k}\\  \end{matrix} \right).
\]
There are two possibilities: either the determinant is identically zero---so that $L_{\hat A^{S,J}, j, J}(\underline{\hat a}_m) \in \varphi_{J,A}H^2(\D)$---or it is not identically zero. In the latter case it means that, up to reordering in increasing way, the multi-index $\{s_1,\ldots, s_{j-1}, m, s_{j+1},\ldots, s_k\}\in \mathcal J_J(\hat A)$. Hence, by definition of $\varphi_{J,A}$, it follows that  $\varphi_{J,A}$ divides the inner part of $L_{\hat A^{S,J}, j, J}(\underline{\hat a}_m)$, that is, $L_{\hat A^{S,J}, j, J}(\underline{\hat a}_m) \in \varphi_{J,A}H^2(\D)$.

In order to show that $T(A)\subseteq W$, let $S=\{1\leq s_1<\ldots<s_k\leq d\}\in\mathcal J_J(\hat A)$ and let
\[
\hat A^{S, J}=\left( \begin{matrix} \hat a_{s_11} & \ldots &  \hat a_{s_1k}\\\vdots&\vdots&\vdots  \\ \hat a_{s_k1} & \ldots &  \hat a_{s_k k}\\  \end{matrix}\right).
\]
Since $\det \hat A^{S, J}\not\equiv 0$ and $\det \hat A^{S, J}\in H^\infty(\D)$, the inner part $\varphi_S:=I(\det \hat A^{S, J})$ of $\det \hat A^{S, J}$ is well defined (and $\varphi_{J,A}$ divides $\varphi_S$ by definition).

We first prove that
\begin{equation}\label{Eq:first-inclusion1}
T_S:=\left(\bigcap_{j=1,\ldots, k} L_{\hat A^{S,J}, j, J}^{-1}(\varphi_S H^2(\D))\right)\cap\mathcal R(A)\subset W.
\end{equation}

Let $f_1\e_1+\ldots +f_d\e_d\in T_S$. Hence,
\begin{equation}\label{Eq:TA-as-reduced-matrix}
 \hbox{rank}\left(\begin{matrix}f_1& \ldots & f_d \\ \hat a_{s_11}& \ldots & \hat a_{s_1d}\\ \vdots& \ldots& \vdots\\ \hat a_{s_k1}  & \ldots & \hat a_{s_kd} \end{matrix}\right)=k.
\end{equation}

 Let $U_S$ be the open dense set of $\D$ such that $\det \hat A^{S, J}(z)\neq 0$. Then, for every $z\in U_S$, by classical linear algebra, there exist $\lambda_{j}(z)\in\C$, $j=1,\ldots, k$ such that
\begin{equation}\label{Eq:sistem-lin-f-j-rank-k}
f_j(z)=\lambda_1(z) \hat a_{s_1j}(z)+\ldots+\lambda_k(z)\hat a_{s_kj}, \quad j=1,\ldots, d.
\end{equation}
Since $\det (\hat A^{S, J}(z))\neq 0$ (for $z\in U_S$), it follows by Cramer's rule that for $j=1,\ldots, k$,
\begin{equation*}
\det(\hat A^{S, J}_S(z)) \lambda_j(z)=\det \left(\begin{matrix} \hat a_{s_11}& \ldots & \hat a_{s_1k}\\ \vdots& \ldots& \vdots\\ \hat a_{s_{j-1}1}& \ldots & \hat a_{s_{j-1}k}\\ f_1& \ldots & f_k  \\ \hat a_{s_{j+1}1}& \ldots & \hat a_{s_{j+1}k}\\  \vdots& \ldots& \vdots\\ \hat a_{s_k1}  & \ldots & \hat a_{s_kk} \end{matrix} \right).
\end{equation*}
Since $\hat a_{jm}\in H^\infty(\D)$ and $f_j\in H^2(\D)$ (for all  indices $j, m$), the right hand side of the previous equation is in $H^2(\D)$ and, since $L_{\hat A^{S,J}, j, J}(f_1\e_1+\ldots+f_d\e_d)\in \varphi_S H^2(\D)$ for $j=1,\ldots, k$ by hypothesis, it follows that it is actually in $\varphi_S H^2(\D)$. Thus, if $O_S\in H^\infty(\D)$ denotes the outer part of $\det \hat A^{S,J}$, we have that $O_S\lambda_j\in H^2(\D)$ for $j=1,\ldots, k$.

Therefore, from \eqref{Eq:sistem-lin-f-j-rank-k}, we have that
\[
O_S(f_1\e_1+\ldots+ f_d\e_d)=\sum_{m=1}^k O_S\lambda_m (\hat a_{s_m1} \e_1+\ldots+\hat a_{s_md}\e_d)\in W.
\]
In particular, it follows that $O_S T_S\subset W$. Taking into account that $T_S$ is a closed, $S$-invariant subspace of $\Ha$ by Lemma~\ref{Lem:continuous-S-L} and Lemma~\ref{lem:r(A)-S-closed-inv}, equation \eqref{Eq:first-inclusion1}  follows at once from Corollary~\ref{Prop:O-closed-S}.

Hence,
\[
P:=\overline{\sum_{S\in\mathcal J_J(\hat A)} T_S}^\Ha\subseteq W.
\]
On the other hand, note that $P$ is closed and $S$-invariant (since it is the closure of the sum of $S$-invariant spaces). Also, note that $\underline{\hat a_m}:=\hat a_{m1}\e_1+\ldots +\hat a_{md}\e_d$ belongs to some $T_S$ (indeed, if $m$ is part of some multi-index $S\in\mathcal J_J(\hat A)$ then $\underline{\hat a_m}\in T_S$, while, if $m$ does not  belong to any multi-index $S\in\mathcal J_J(\hat A)$ then $L_{\hat A^{S,J}, j, J}(\underline{\hat a_m})\equiv 0$ for all $S\in\mathcal J_J(\hat A)$ and $j\in\{1,\ldots, k\}$, so that $\underline{\hat a_m}\in T_S$ for all $S\in\mathcal J_J(\hat A)$). Thus, $W\subseteq P$ and therefore
\[
P=W.
\]
We are thus left to show that
\[
T(A)\subseteq P.
\]
By Corollary~\ref{Cor:conten-chiuso-S}, it is enough to show that $T(A)\cap \Ha_\infty\subseteq P$. To this aim, let $F:=f_1\e_1+\ldots+f_d\e_d\in T(A)\cap \Ha_\infty$. For every $S\in\mathcal J_J(\hat A)$ let $\theta_S$ be the inner function (unique up to multiplication by a unimodular constant) such that $\varphi_S=\varphi_{J,A}\theta_S$. Note that the (inner function) greatest common divisor of $\{\theta_S\}_{S\in\mathcal J_J(\hat A)}$ is $1$ (for otherwise $\varphi_{J,A}$ would not be the greatest common divisor of the $\varphi_S$'s).

Fix $S\in \mathcal J_J(\hat A)$ and $j\in\{1,\ldots, k\}$. Since by hypothesis $L_{\hat A^{S,J}, j, J}(F)\in \varphi_{J,A}H^2(\D)$, it follows by the (multi-)linearity of the determinant that
\[
L_{\hat A^{S,J}, j, J}(\theta_S F)\in \varphi_S H^2(\D).
\]
Thus, $\theta_S F\in P$ and, in particular, since $F\in\Ha_\infty$,  we have $h\theta_S F\in P$ for all $h\in H^2(\D)$.
Hence, if for each $S\in\mathcal J_J(\hat A)$ we take a function $h^S\in H^2(\D)$, we have
\[
\left(\sum_{S\in\mathcal J_J(\hat A)} h^S\theta_S\right) F \in P.
\]
Now, let $\{h^S_n\}\subset H^2(\D)$ be the sequences given by  Lemma~\ref{Lem:coprime-corona} such that $\{\sum_{S\in\mathcal J_J(\hat A)} h_n^S\theta_S\}$ converges to $1$ in $H^2(\D)$, and let
\[
F_n:=\left(\sum_{S\in\mathcal J_J(\hat A)} h_n^S\theta_S\right) F\in P.
\]
Taking into account that $F\in \Ha_\infty$, it follows that $\{F_n\}$ converges to $F$ in $\Ha$ and hence $F\in P$, and we are done.
\end{proof}

\end{document}